\documentclass{amsart}

\title{The Norm map of Witt vectors}
\author{Vigleik Angeltveit}
\address{Mathematical Sciences Institute \\
John Dedman Building (Building 27) \\
Australian National University \\
Acton, ACT \\
0200 Australia}

\usepackage{amsxtra}
\usepackage{amsfonts}
\usepackage[latin1]{inputenc}
\usepackage{graphicx}
\usepackage{amsmath,amssymb,latexsym,amsthm,mathrsfs}
\usepackage[all]{xy}
\usepackage{pstricks}
\usepackage{verbatim}

\newtheorem{theorem}{Theorem}[section]
\newtheorem{thm}[theorem]{Theorem}
\newtheorem{lemma}[theorem]{Lemma}

\theoremstyle{definition}
\newtheorem{defn}[theorem]{Definition}
\newtheorem{remark}[theorem]{Remark}

\makeatletter
\let\c@equation\c@theorem
\makeatother
\numberwithin{equation}{section}

             \newcommand{\bN}{\mathbb{N}}  
      \newcommand{\bW}{\mathbb{W}}   \newcommand{\bZ}{\mathbb{Z}}

\newcommand{\xto}{\xrightarrow}

\begin{document}
 
\begin{abstract}
We discuss a multiplicative version of the Verschiebung map of Witt vectors that we call the \emph{norm}.
\end{abstract}

\maketitle

\section{Introduction}
Witt vectors are ubiquitous in algebraic geometry and number theory, and also appear in algebraic topology when studying fixed points of the topological Hochschild homology spectrum. Most of the structure maps that are present have been exploited with great success, see e.g.\ Hesselholt and Madsen's paper \cite{HeMa04}.

In this short note we discuss a structure map of Witt vectors that appears implicitly in \cite{Br05} but to the best of our knowledge has never been made explicit. We also briefly discuss the connection with topological Hochschild homology. For a different perspective on how the norm map fits with the rest of the Witt vector structure maps, see \cite{An_genWitt}. We hope to return to further applications of the norm map elsewhere.

Let $k$ be a commutative ring and let $\bW(k)$ denote the ring of big Witt vectors of $k$. Fix a prime $p$. Recall (e.g.\ from \cite[1.4.8]{Il79}) that the Frobenius $F_p : \bW(k) \to \bW(k)$ satisfies
\[
 F_p(a) \equiv a^p \mod p,
\]
and define $\theta_p : \bW(k) \to \bW(k)$ by requiring that the equation
\[
 F_p(a) = a^p + p \theta_p(a)
\]
holds functorially in the ring $k$.

\begin{defn} \label{d:norm}
For a prime $p$, define $N_p : \bW(k) \to \bW(k)$ by
\[
 N_p(a) = a - V_p \theta_p(a).
\]
For a composite integer $d$, define $N_d$ inductively by factoring $d=d_1 \cdot d_2$ and setting $N_d = N_{d_1} \circ N_{d_2}$.
\end{defn}

Recall that a \emph{truncation set} is a set $S \subset \bN=\{1,2,\ldots\}$ which is closed under division, and that given a truncation set $S$ it is possible to define the Witt vectors $\bW_S(k)$ (see e.g.\ \cite[\S 3]{HeMa97a}). Given $d \in \bN$, let $\langle d \rangle$ denote the set of divisors of $d$. Given a truncation set $S$, let
\[
 T = \langle d \rangle S = \{t \in \bN \quad | \quad t=es \textnormal{ for some $e \mid d$, $s \in S$}\}.
\]
Note that with this definition $S = T/d$, where $T/d$ is defined as in Section \ref{s:background} below. We will prove the following result:

\begin{thm} \label{t:main1}
The map $N_d$ is a multiplicative map. Moreover, if $S$ is any truncation set then $N_d$ restricts to a multiplicative map
\[
 N_d : \bW_S(k) \to \bW_{\langle d \rangle S}(k),
\]
and the composite $F_d \circ N_d$ is the $d$'th power map $a \mapsto  a^d$.
\end{thm}

This result is perhaps surprising, because $N_d$ does not look multiplicative and in the definition of $N_p$ as $N_p = id - V_p \theta_p$ both $id$ and $V_p \theta_p$ map $\bW_S(k)$ to itself and $\bW_{\langle d \rangle S}(k)$ is larger. 

One natural question is to what extent the map $N_d$ is unique. We answer that as follows.

\begin{thm} \label{t:main2}
Suppose $N : \bW(k) \to \bW(k)$ is multiplicative, natural in $k$, restricts to a map $\bW_S(k) \to \bW_{\langle d \rangle S}(k)$ for any truncation set $S$, and suppose $F_d \circ N$ is the $d$'th power map. Then $N=N_d$.
\end{thm}

The key to proving Theorems \ref{t:main1} and \ref{t:main2} is to understand what happens on ghost coordinates. In fact the formula for $N_d$ on ghost coordinates is very simple, as our next result shows.

\begin{thm} \label{t:main3}
On ghost coordinates the map $N_d$ is given by $\langle x_s \rangle \mapsto \langle y_t \rangle$ with
\[
 y_t = x_{t/g}^g \quad \textnormal{with} \quad g=\gcd(d,t).
\]
\end{thm}

For example, consider the case when $d=p$ is prime and $S=\{1,p,\ldots,p^{n-1}\}$. Then $\langle p \rangle S = \{1,p,\ldots,p^n\}$, and Theorem \ref{t:main1} gives a map
\[
 N_p : W_{n-1}(k;p) \to W_n(k;p)
\]
which by Theorem \ref{t:main3} is given on ghost coordinates by
\[
 N_p\langle x_0,x_1,\ldots,x_{n-1} \rangle = \langle x_0, x_0^p, x_1^p,\ldots,x_{n-1}^p \rangle.
\]

\begin{remark}
The relation $F_d \circ N_d(a)=a^d$ is analogous to the relation $F_d \circ V_d(a)=da$ between the Verschiebung and Frobenius maps. But while the Verschiebung map $V_d : \bW_S(k) \to \bW_T(k)$ is defined for any truncation sets $S$ and $T$ with $S=T/d$, we need the stronger condition that $T=\langle d \rangle S$ for the norm map.

To see this, consider the following example. Let $T=\{1,2,3\}$ and let $d=2$. Then $S = T/2=\{1\}$. But then there is no way to define a map $N : \bW_S(k) \to \bW_T(k)$ such that the diagram
\[ \xymatrix{
 \bW(k) \ar[r]^-{N_2} \ar[d]_{R^\bN_S} & \bW(k) \ar[d]^{R^\bN_T} \\
 \bW_S(k) \ar[r]^-N & \bW_T(k)
} \]
commutes, because on ghost coordinates $R^\bN_T \circ N_2$ is given by $\langle x_n \rangle \mapsto \langle x_1, x_1^2, x_3 \rangle$ and this does not factor through $\bW_S(k)$.

(This is not a problem when defining the Verschiebung as the map given on ghost coordinates by $\langle x_n \rangle \mapsto \langle 0, 2x_1,0 \rangle$ does factor through $\bW_S(k)$.)
\end{remark}

\subsection{Relations to previous work}
The results by Brun \cite{Br05} give a norm map, but in a more roundabout way. Given a finite group $G$ and a commutative ring $k$, Brun regards $k$ as a functor from the category $U^{fG}$ to sets, where the objects of $U^{fG}$ are finite free $G$-sets and the morphisms are bispans of such. For formal reasons there is a functor $Fun(U^{fG}, Set) \to Fun(U^G,Set)$, where $U^G$ is the category of all finite $G$-sets and bispans. Hence $k$ gives rise to a $G$-Tambara functor.

His result \cite[Theorem B]{Br05} then says that evaluating this $G$-Tambara functor on $G/H$ gives $\bW_H(k)$, where $\bW_H(k)$ denotes the generalized Witt vectors of Dress and Siebeneicher \cite{DrSi88}. Because a $G$-Tambara functor has norm maps this then gives a norm map $N_H^G : \bW_H(k) \to \bW_G(k)$.

We can put ourselves in a situation where the norm maps can be compared as follows. Given natural numbers $n$ and $d$, let $G=C_{dn}$ and $H=C_n$. Also let $S=\langle n \rangle$ and $T=\langle dn \rangle$. Then $\bW_H(k)=\bW_S(k)$ and $\bW_G(k)=\bW_T(k)$, so our norm map $N_d : \bW_S(k) \to \bW_T(k)$ and the norm map $N_H^G : \bW_H(k) \to \bW_G(k)$ implicit in Brun's paper have the same source and target. It follows from Theorem \ref{t:main2} that they must agree.

\subsection{Acknowledgements}
The author would like to thank James Borger and Arnab Saha for interesting conversations about Witt vectors. The formula $N_p(a)=a-V_p \theta_p(a)$ for the norm is due to James Borger.

\section{Background on Witt vectors} \label{s:background}
In \cite{Wi37} Witt defined the $p$-typical Witt vectors as a way of lifting commutative rings from characteristic $p$ to characteristic $0$. Later Cartier \cite{Ca67} generalized Witt's construction to what has become known as the big Witt vectors. For a summery of the basic properties of Witt vectors, see e.g.\ \cite[\S 3]{HeMa97a} or Hesselholt's survery article \cite{He}. Here we give a very short summary of some of the properties of Witt vectors.

Given a truncation set $S$ and a commutative ring $k$, the ring $\bW_S(k)$ of Witt vectors is defined to be $k^S$ as a set, and addition and multiplication are defined by the requirement that the ghost map $w : \bW_S(k) \to k^S$ defined by
\[
 w(a)_s = \sum_{d \mid s} da_d^{s/d}
\]
is a ring homomorphism, functorially in $k$. We will write a Witt vector as $a=(a_s)$ and the image of a Witt vector under the ghost map as $x=\langle x_s \rangle$.

Given $T \subset S$ we have a restriction map $R^S_T : \bW_S(k) \to \bW_T(k)$ defined on Witt coordinates by $(R^S_T(a))_t=a_t$, or equivalently on ghost coordinates by $(R^S_T \langle x \rangle)_t = x_t$. Given a natural number $d$ and a truncation set $S$, it is customary to define
\[
 S/d = \{n \in \bN \quad | \quad dn \in S\}.
\]
Then the Frobenius $F_d : \bW_S(k) \to \bW_{S/d}(k)$ is defined on ghost coordinates by $(F_d \langle x \rangle)_n = x_{dn}$ while the formula on Witt coordinates is more complicated. The Verschiebung can be defined either on Witt coordinates by $(V_d(a))_s = a_{s/d}$ if $d \mid s$ and $0$ otherwise, or on ghost coordinates by $(V_d \langle x \rangle)_s = dx_{s/d}$ if $d \mid s$ and $0$ otherwise. These satisfy $F_d \circ V_d = d$, while $F_d \circ V_e = V_e \circ F_d$ for $\gcd(d,e)=1$. The composite $V_d \circ F_d$ is multiplication by $V_d(1)$. Both $F_d$ and $V_d$ commute with restriction maps, in the sense that the two superimposed diagrams
\[ \xymatrix{
 \bW_{S/d}(k) \ar@<0.7ex>[r]^-{V_d} \ar[d]_{R^{S/d}_{T/d}} & \bW_S(k)n \ar@<0.7ex>[l]^-{F_d} \ar[d]^{R^S_T} \\
 \bW_{T/d}(k) \ar@<0.7ex>[r]^-{V_d} & \bW_T(k) \ar@<0.7ex>[l]^-{F_d}
} \]
commute for any $T \subset S$.

From Theorem \ref{t:main3} it follows that the norm map $N_d$ commutes with restriction maps in the sense that the diagram
\[ \xymatrix{
 \bW_S(k) \ar[r]^-{N_d} \ar[d]_{R^S_T} & \bW_{\langle d \rangle S}(k) \ar[d]^{R^{\langle d \rangle S}_{\langle d \rangle T}} \\
 \bW_T(k) \ar[r]^-{N_d} & \bW_{\langle d \rangle T}(k)
} \]
commutes for any $T \subset S$.

\section{Proofs}
In addition to proving Theorems \ref{t:main1}, \ref{t:main2} and \ref{t:main3} we need to prove that Definition \ref{d:norm} is independent of the factorisation of $d$. We start with the description of the norm map on ghost coordinates.

\begin{proof}[Proof of Theorem \ref{t:main3}]
Suppose first that $d=p$ is prime. We start by computing $N_p(a)=a-V_p \theta_p(a)$ on ghost coordinates. To begin with, $\theta_p\langle x \rangle$ is given by
\[
 (\theta_p \langle x \rangle)_s = \Big( \frac{F_p \langle x \rangle - \langle x \rangle^p}{p} \Big)_s = \frac{x_{ps}-x_s^p}{p}.
\]
We then get
\[
 (N_p \langle x \rangle)_s = x_s - \big(V_p\big\langle t \mapsto \frac{x_{pt}-x_t^p}{p} \big\rangle\big)_s = x_s - \begin{cases} x_s - x_{s/p}^p \quad & \textnormal{if } p \mid s \\ 0 \quad & \textnormal{if } p \nmid s \end{cases}
\]
so
\[
 (N_p \langle x \rangle)_s = \begin{cases} x_{s/p}^p \quad & \textnormal{if } p \mid s \\ x_s \quad & \textnormal{if } p \nmid s \end{cases}
\]

Now suppose $d = d_1 \cdot d_2$ and that $N_{d_1}$ and $N_{d_2}$ are given on ghost coordinates by the formula in Theorem \ref{t:main3}. Using Definition \ref{d:norm} we get
\[
 (N_d \langle x \rangle)_s = (N_{d_1} \circ N_{d_2} \langle x \rangle)_s = (N_{d_2} \langle x \rangle)_{s/g}^g = (x_{(s/g)/h}^g)^h = x_{s/gh}^{gh},
\]
where $g=\gcd(d_1,s)$ and $h=\gcd(d_2,s/g)$. But then $gh=\gcd(d,s)$. This provides the induction step to proving Theorem \ref{t:main3}, as well as proving that the definition of $N_d$ in Definition \ref{d:norm} is independent of the factorization of $d$.
\end{proof}

With the formula for $N_d$ on ghost coordinates in hand the proof of Theorem \ref{t:main1} is easy.

\begin{proof}[Proof of Theorem \ref{t:main1}]
It follows immediately from the fact that $N_d$ is multiplicative on ghost coordinates that $N_d$ is multiplicative. Moreover, if $S$ is any truncation set and $t \in \langle d \rangle S$ then $(N_d \langle x_n \rangle)_t$ depends only on $x_s$ for $s \in S$. It then follows that $N_d$ restricts to a multiplicative map $\bW_S(k) \to \bW_{\langle d \rangle S}(k)$ because it does so on ghost coordinates.

Finally, the composite $F_d \circ N_d$ is given on ghost coordinates by $\langle x_s \rangle \mapsto \langle x_s^d \rangle$. Since the composite is the $d$'th power map on ghost coordinates it must also be so on Witt coordinates.
\end{proof}

Finally we prove uniqueness.

\begin{proof}[Proof of Theorem \ref{t:main2}]
We can assume without loss of generality that $d=p$ is prime. Suppose we have a map
\[
 N : \bW(k) \to \bW(k)
\]
satisfying the conditions in Theorem \ref{t:main2}, and suppose $N \neq N_p$. Let $N_p(a_n)=(b_n)$, $N(a_n)=(b_n')$, and on ghost coordinates $N_p \langle x_n \rangle = \langle y_n \rangle$, $N \langle x_n \rangle = \langle y_n' \rangle$. Then $y_n=x_n$ if $p \nmid n$ and $y_n=y_n' = x_{n/p}^p$ if $p \mid n$.

Because $N$ is multiplicative, it follows that $y_n'$ is a monomial in the variables $x_m$ with constant coefficient $1$ and because $N$ factors as $\bW_S(k) \to \bW_{\langle p \rangle S}(k)$ it follows that $y_n'$ only depends on $x_m$ for $m \mid n$.

Because $N$ is natural in the ring $k$, we can work with $k=\bZ[a_n]_{n \in \bN}$ and let $(a_n)$ be the canonical Witt vector in $\bW(k)$. Since we assumed $N \neq N_p$ there must be some smallest $t \in \bN$ with $b_t \neq b_t'$, and it follows that $p \nmid t$. Because $k$ is torsion free it follows that $y_t \neq y_t'$ as well.

Now we use that $y_{pt} = y_{pt}'$. By comparing
\[
 y_{pt} = \sum_{d \mid t} d b_d^{pt/d} + \sum_{d \mid t} pd b_{pd}^{t/d}
\]
and
\[
 y_{pt}' = \sum_{d \mid t} d (b_d')^{pt/d} + \sum_{d \mid t} pd (b_{pd}')^{t/d}
\]
and using that $b_d=b_d'$ for $d<t$, we find that $b_t \equiv b_t' \mod p$ and hence $y_t \equiv y_t' \mod p$. But this suffices to conclude that $y_t=y_t'$ because both are monomials with constant coefficient $1$ in $x_m$ for $m \mid t$.
\end{proof}

\section{Relation to the norm in equivariant stable homotopy theory}
Recall e.g.\ from \cite{HeMa97a} that for any ring $k$, the \emph{topological Hochschild homology} spectrum $THH(k)$ is a \emph{cyclotomic spectrum}, meaning a genuine $S^1$-equivariant spectrum with restriction maps $R_d : THH(k)^{C_{dn}} \to THH(k)^{C_n}$. If $k$ is commutative then $THH(k)$ comes with a multiplication map, and can be made into a commutative $S$-algebra in the category of equivariant orthogonal spectra. From now on let us once again assume that $k$ is commutative. From \cite[Addendum 3.3]{HeMa97a} we have an isomorphism
\[
 \pi_0 THH(k)^{C_n} \xto{\cong} \bW_{\langle n \rangle}(k)
\]
of rings, where as before $\langle n \rangle$ denotes the truncation set of divisors of $n$.

From the results in \cite{GrMa97} it follows that there is a norm map
\[
 N_d^{top} : \pi_0 THH(k)^{C_n} \to \pi_0 THH(k)^{C_{dn}}.
\]

\begin{thm}
There is a commutative diagram
\[ \xymatrix{
 \pi_0 THH(k)^{C_n} \ar[r]^-\cong \ar[d]_{N_d^{top}} & \bW_{\langle n \rangle}(k) \ar[d]^{N_d} \\
 \pi_0 THH(k)^{C_{dn}} \ar[r]^-\cong & \bW_{\langle dn \rangle}(k)
} \]
relating the Witt vector norm map to the topological norm map on $THH(k)$.
\end{thm}

\begin{proof}
This follows from Theorem \ref{t:main2} because the topological norm map also satisfies $F_d \circ N_d^{top}(a)=a^d$. Here $F_d : \pi_0 THH(k)^{C_{dn}} \to \pi_0 THH(k)^{C_n}$ is induced by inclusion of fixed points.
\end{proof}

One can also obtain a statement about more general truncation sets by considering
\[
 \lim_{n \in S} \pi_0 THH(k)^{C_n},
\]
with the maps in the diagram given by restriction. We omit the details.

\section{Relations}
We have already described some of the relations between the norm map and the other structure maps of Witt vectors. Theorem \ref{t:main1} says that
\[
 F_d \circ N_d (a) = a^d.
\]

\begin{lemma}
If $\gcd(d,e)=1$ then
\[
 F_e \circ N_d = N_d \circ F_e.
\]
\end{lemma}

\begin{proof}
This follows immediately from the corresponding statement on ghost coordinates, and this is easily verified using the explicit formula for $N_d$ from Theorem \ref{t:main3} and the corresponding formula defining $F_e$.
\end{proof}

Note that these two relations suffice for describing $F_e \circ N_d$ in general. We described the relationship between $N_d$ and restriction maps earlier. For completeness we restate that relation here:
\[
 N_d \circ R^S_T = R^{\langle d \rangle S}_{\langle d \rangle T} \circ N_d.
\]

The norm map is of course not additive, but we can say the following.

\begin{lemma}
Suppose $p$ is prime. We have
\[
 N_p(a + b) = N_p(a) + N_p(b) + \sum_{i=1}^{p-1} \frac{1}{p} \binom{p}{i} V_p(a^i b^{p-i}).
\]
\end{lemma}

\begin{proof}
This follows from the corresponding formula on ghost coordinates, starting with
\[
 (V_p \langle x^i y^{p-i} \rangle )_n = \begin{cases} p x_{n/p}^i y_{n/p}^{p-i} \quad & \textnormal{if $p \mid n$} \\ 0 \quad & \textnormal{if $p \nmid n$} \end{cases}
\]
The verification is straightforward and will be omitted.
\end{proof}

The description of how to commute the norm past the Verschiebung commute is similar:

\begin{lemma}
Suppose $p$ is prime and $\gcd(p,q)=1$. Then
\[
 N_p \circ V_q(a) = V_q \circ N_p (a) + \frac{q^p-q}{pq} V_{pq}(a^p).
\]
If $p=q$ we have
\[
 N_p \circ V_p(a) = p^{p-2} V_{p^2}(a^p).
\]
\end{lemma}

\begin{proof}
These formulas can easily be verified in ghost coordinates, and because they are valid there they must be valid in Witt coordinates as well.
\end{proof}

Of course other relations hold as well. For exmple we can say the following:

\begin{lemma}
Suppose $k$ contains a primitive $p$'th root of unity $\xi$. Then
\[
 \sum_{i=0}^{p-1} N_p(\xi^i a) = V_p(a^p).
\]
\end{lemma}

For example we always have $N_2(a)+N_2(-a)=V_2(a^2)$.

\begin{proof}
Once again this can be easily verified in ghost coordinates.
\end{proof}

\bibliographystyle{plain}
\bibliography{b.bib}

\end{document}